\documentclass[11pt]{amsart}
\usepackage[margin=1in]{geometry} 
\usepackage[dvipsnames,usenames]{color}
\usepackage{hyperref}
\usepackage{cleveref}
\usepackage{graphicx}
\usepackage{epsfig}
\usepackage[latin1]{inputenc}
\usepackage{amsmath,amsfonts,amssymb,amsthm,amscd}
\usepackage{verbatim}
\usepackage{subfigure}
\usepackage{pinlabel}
\usepackage{stmaryrd}
\usepackage{enumerate,enumitem,array}
\usepackage{todonotes}
\usepackage{bm}
\usepackage{thmtools}
\usepackage{thm-restate}
\usepackage{comment}
\usepackage{tikz}
\usepackage{cancel}

\usepackage{mathtools}
\usepackage{xypic}

\newtheorem{theorem}{Theorem}[section]

\newtheorem{lemma}[theorem]{Lemma}
\newtheorem{proposition}[theorem]{Proposition}

\newtheorem{corollary}[theorem]{Corollary}

\theoremstyle{definition}

\theoremstyle{remark}
\newtheorem{remark}[theorem]{Remark}

\newcommand{\Z}{\mathbb{Z}}

\newcommand{\R}{\mathbb{R}}
\newcommand{\mfs}{\mathfrak{s}}
\newcommand{\mfsbar}{\overline{\mfs}}
\newcommand{\mft}{\mathfrak{t}}

\newcommand{\Sol}{M_{a,b}}
\newcommand{\Solprime}{M_{a',b'}}

\newcommand{\solextend}{\mathcal{S}_{b,a}}
\newcommand{\solextenda}{\mathcal{S}_{a}}
\newcommand{\solextendb}{\mathcal{S}_{b}}
\newcommand{\solnoextend}{\mathcal{S}_{\emptyset}}

\newcommand{\wa}{\iota_{W_a}}
\newcommand{\wb}{\iota_{W_{-b}}}
\newcommand{\da}{\iota_{D_a}}
\newcommand{\db}{\iota_{D_{-b}}}
\DeclareMathOperator{\Ima}{Im}

\title{Homology cobordism for some Sol manifolds}

\author{Tye Lidman}
\address{North Carolina State University}
\email{tlid@math.ncsu.edu}
\thanks{TL was supported by NSF grants DMS-2105469, DMS-2506277, and a Simons Travel grant}

\author{Juanita Pinz\'{o}n-Caicedo}
\address{University of Notre Dame}
\email{jpinzonc@nd.edu}
\thanks{JPC was partially supported by a Simons Travel  grant}

\begin{document}
\maketitle 

\begin{abstract}
We show that if $Y_1,Y_2$ are 3-manifolds that admit Sol geometry and have first homology group of order 16, then $Y_1,Y_2$ are integer homology cobordant if and only if they are homeomorphic.
\end{abstract}

\section{Introduction}

The existence of homology cobordisms between 3-manifolds has been an important problem in geometric topology for many decades. Indeed, Galewski-Stern~\cite{gs} and Matumoto~\cite{matumoto} reduced the high-dimensional Triangulation Conjecture to the existence of order 2 elements in the (3+1)-dimensional homology cobordism group.  Manolescu~\cite{manolescu} completed this program by disproving the high-dimensional Triangulation Conjecture using a homology cobordism invariant he developed from $Pin(2)$-equivariant Seiberg-Witten Floer theory.  In general, tools originating from gauge theory and or Floer theory have been quite useful for studying problems in homology cobordism.  For example, Fintushel-Stern~\cite{fs} carried out an integer homology cobordism classification of lens spaces using $SO(3)$ moduli spaces of self-dual and anti-self-dual connections. Doig and Wehrli~\cite{doig-wehrli} gave an alternate proof more recently using the $d$-invariants from Heegaard Floer homology.  While the full classification of all 3-manifolds according to homology cobordism seems out of reach, restricting the classification to subsets of 3-manifolds with constrained geometric types seems plausible.  One could for example consider only 3-manifolds with spherical geometry, and with non-cyclic fundamental group. These can be further separated into icosahedral, tetrahedral, octahedral, or dihedral. The problem is trivial for the first three types, since there is at most one manifold with fixed order of first homology. For the case of dihedral manifolds, much can be deduced from the work of Doig~\cite{doig}, who gives a recursive formula for the Heegaard Floer $d$-invariants, and the work of Ballinger-Ni-Ochse-Vafaee~\cite{prism1,prism2,prism3}, who describe the class of dihedral manifolds realizable by surgeries on knots in $S^3$. One could then expand their focus to consider Seifert fibered manifolds in general, or even hyperbolic manifolds. At the time of writing, the former remains highly incomplete, and the latter almost perfectly intractable. We thus instead focus on the last of the Thurston geometries: we classify integer homology cobordism for the Sol manifolds with smallest first homology, order 16. To emphasize how important it is to work with $\Z$-homology, we observe below (see \Cref{sec:rational-sol-dihedral}) that every Sol rational homology sphere with $|H_1| = 16$ bounds a rational homology ball!  Our main result is the following:
\begin{theorem}\label{thm:main}
Let $Y_1$ and $Y_2$ be two Sol rational homology spheres with $|H_1(Y_i)| = 16$.  Then, $Y_1$ and $Y_2$ are integrally homology cobordant if and only if they are orientation-preserving homeomorphic.   
\end{theorem}

Our strategy to prove Theorem~\ref{thm:main} can be described as follows. First, for $Y$ a Sol rational homology sphere satisfying $|H_1(Y)| = 16$, we consider the set of $d$-invariants of $Y$, namely, the image of the function $d_{Y}:Spin^c(Y_i) \to \mathbb{Q}$ that assigns to each $\mft\in Spin^c(Y)$ the grading of the generator of the Heegaard Floer group $\widehat{HF}(Y,\mft)$. Then, if $Y'$ is another Sol rational homology sphere satisfying $|H_1(Y')| = 16$, we show that the functions $d_Y$ and $d_{Y'}$ are necessarily different. We do so after exhibiting rational homology cobordisms between $Y$ and two dihedral manifolds of the form $D_{-b}=S^2\left (0;(2,1),(2,-1) ,(b,-1)\right )$. Paired with Doig's results from \cite{doig}, the existence of these rational homology cobordisms allows us to compute 12 of the 16 $d$-invariants of $Y$.  We then rely on a careful analysis of the structural properties of the set of spin$^c$ structures, their first Chern classes, and their $d$-invariants to uniquely determine the Sol manifold.  While we believe a version of this general strategy should extend to Sol rational homology spheres with larger first homology, we do not pursue this here. The main issue is that larger homology groups for Sol manifolds will be paired via rational homology cobordism to dihedral manifolds with commensurately larger first homology, and the $d$-invariants of these are harder to work with.


\section{Background on Sol and dihedral manifolds}
Our goal is to understand the $d$-invariants of Sol manifolds.  While we do not completely compute these, we can still extract a great deal of information. In order to do this, we need to relate Sol manifolds to dihedral manifolds, so we begin with a review of these simpler spaces.  

\subsection{Dihedral manifolds}\label{dihedral}
A 3-manifold is called spherical if it can be obtained as a quotient of $S^3$ by a finite subgroup $\Gamma$ of $SO(4)$, and all spherical manifolds are L-spaces \cite{OSLens, KMOS}. If the group $\Gamma$ is a central extension of a dihedral group, then the manifold is called dihedral\footnote{Some authors also use the term prism to describe these manifolds.}. These  manifolds admit two different types of Seifert fibrations, one with base $\R P^2$ and another with base $S^2$, and are examples of closed and oriented Seifert fibered spaces with finite but noncyclic fundamental group. See \cite[5.4 and 6.2]{orlik} for more details. For the purposes of this article, it is enough to focus on the case of base $S^2$. Namely, a dihedral manifold will be the Seifert fibered manifold, shown in Figure~\ref{fig:dihedral}, with the following Seifert invariants:%
\begin{equation}
D_{-b/c} = S^2\left (0;(2,1),(2,-1) ,(b,-c)\right ).\footnote{When $|b| = 1$, the manifold $D_{-b}$ is a lens space.  The actual geometric type of this manifold is not so important for the proof, so for ease of exposition we view these as degenerate dihedral manifolds.}
\end{equation}
Notice that in this case we have $e\left(D_{-b/c}\right)=c/b$. The following lemma is well-known, but we include it for future reference.
\begin{lemma}[Seifert]\label{prop:H1(dih)} Let $D_{-b/c}= S^2\left (0;(2,1),(2,-1) ,(b,-c)\right )$.  Then $$H_1(D_{-b/c})=\begin{cases}
\mathbb{Z}/2 \oplus \mathbb{Z}/2c & \text{ if } b\equiv 0 \pmod{2},\\
\mathbb{Z}/4c & \text{ otherwise.}
\end{cases}$$
\end{lemma}
\begin{proof} Following \cite[6.2]{neumann}, using the Seifert fibration $D_{-b/c}= S^2\left (0;(2,1),(2,-1) ,(b,-c)\right )$, one computes 
$H_1(D_{-b/c})$ as the cokernel of the matrix $$\left(\begin{array}{rrrr}1 & 1 & 1 & 0 \\2 & 0 & 0 & 1 \\0 & -2 & 0 & 1 \\0 & 0 & -b & c\end{array}\right).$$ 
See \Cref{prop:cobs} for a computation using a different decomposition of $D_{-b/c}$.
\end{proof}

The proof of our main result relies on specific computations for the $d$-invariants of dihedral manifolds. We will only be interested in the case $c = 1$, and so for an integer $n$, we use $D_{-n}$ to denote the manifold with Seifert invariants $S^2\left (0;(2,1),(2,-1) ,(n,-1)\right)$. By an abuse of notation, we will use $D_0$ to denote $\R P^3\# \overline{\R P^3}$, which is consistent with \Cref{fig:dihedral}.  

\begin{figure}[h]
\centering
\def\svgwidth{0.325\textwidth}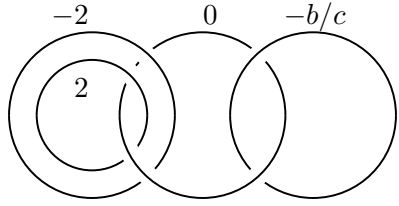
\caption{Surgery diagram for the dihedral manifold $D_{-b/c}= S^2\left (0;(2,1),(2,-1) ,(b,-c)\right )$.}\label{fig:dihedral}
\end{figure}

\subsection{Sol}
The Lie group $Sol$ consists of the pair $(\R^3,\star)$ where $\star$ is the binary operation defined by $$\vec{v}\star\vec{x}=(e^{v_3}x_1+v_1,e^{-v_3}x_2+v_2,x_3+v_3).$$ A 3-manifold $M$ admits a $Sol$ geometry if it can be realized as a quotient $Sol/\Gamma$ for $\Gamma$ a subgroup of $Isom(Sol)$. In this article we will use a more topological description of rational homology spheres that admits a Sol geometry, which we now describe.  Let $N$ be the twisted $I$-bundle over the Klein bottle. Two copies of $N$ glued together by an orientation-reversing self-diffeomorphism $\phi$ of the 2-torus $\partial N$ is a closed orientable 3-manifold $S_\phi$ sometimes referred to as a torus semi-bundle. The manifold $N$ can be described in the following ways:
\begin{enumerate}
\item the exterior of a knot $J$ in $\R P^3\#\overline{\R P^3}$ as shown in \Cref{fig:I-bundle}, 
\item a Seifert fibered space with base orbifold $D^2(2,2)$,
\item a Seifert fibered space with base orbifold a M\"obius band and no cone points.
\end{enumerate}

Using the description in \Cref{fig:I-bundle}, we obtain a fixed orientation on $N$, and hence an orientation of $S_\phi$.  If the gluing $\phi$ identifies fibers of either of the two Seifert fibrations, then the 3-manifold $S_\phi$ is Seifert fibered. In any other case, $S_\phi$ admits a Sol geometry. To be more precise, denote by $\{\mu,\lambda\}$ the meridional-longitude pair of the knot $J$ regarded as a knot in $S^3$, and use the ordered set to get a basis for $H_1(\partial N)$. The curves $\mu,\lambda$ are precisely the fibers of the two Seifert structures on $N$, and so the map $\phi$ can be identified with a matrix $A_\phi=\left(\begin{smallmatrix}a & c \\d & b\end{smallmatrix}\right)$ with determinant $-1$.  Whenever $a,b,c,d\neq 0$, the matrix $A_\phi$ defines a Sol manifold, and in fact every oriented Sol rational homology sphere can be described in this way (see \cite[Theorem 5.2]{ScottGeometries} and \cite[Section 9]{BoyerRolfsenWiest}).  In addition, this choice of basis for $H_1(\partial N)$ can be extended to a basis for $H_1(N)$, and with that, a matrix presentation for the map $\iota: H_1(\partial N)\to H_1(N)$ induced by inclusion.  More precisely, we have
\begin{equation}\label{H1(N)}
H_1(N)=\langle y, \mu,\lambda \mid 2\lambda=0,\; 2y+\mu=0\rangle = \langle y, \lambda \mid 2\lambda=0\rangle = \Z\oplus\Z/2,
\end{equation}
where $y$ is the meridian of the 0-framed component and so $\iota$ has matrix presentation $\left(\begin{smallmatrix}-2 & 0 \\\phantom{-}0 & 1\end{smallmatrix}\right)$.  
\begin{figure}[h]
\centering
\def\svgwidth{0.1375\textwidth}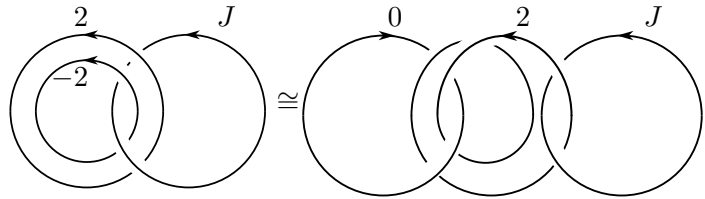
\caption{The twisted $I$-bundle over the Klein bottle represented as the exterior of the knot $J$ in $\R P^3\#\overline{\R P^3}$.}\label{fig:I-bundle}
\end{figure}

Our proof of \Cref{thm:main} involves a careful analysis of the Spin$^c$ structures of Sol rational homology spheres. Since these correspond (non-canonically) with the first homology group with integer coefficients, we start with the following computation:
\begin{lemma}\label{prop:H1(sol)} Let $S_\phi$ be a Sol rational homology sphere determined by a gluing matrix $A_\phi=\left(\begin{smallmatrix}a & c \\d & b\end{smallmatrix}\right)$ with $\det\left(A_\phi\right)=-1$. Then $$H_1(S_\phi)=\begin{cases}
\mathbb{Z}/2 \oplus \mathbb{Z}/2 \oplus \mathbb{Z}/4c & \text{ if } d\equiv 0 \pmod{2}, \\
\mathbb{Z}/4 \oplus \mathbb{Z}/4c & \text{ otherwise.}
\end{cases}$$
\end{lemma}

\begin{proof}
The Mayer-Vietoris sequence applied to the decomposition $S_\phi=N_1\cup_\phi N_2$ shows that $H_1(S_\phi)$ is the cokernel of the map $H_1(\partial N_2)\to H_1(N_1)\oplus H_1(N_2)$ given by $x\to (\iota_1\circ\phi(x),\iota_2(x))$, where $\iota_j$ denotes the map on first homology induced by inclusion. The meridian-longitude pair of the knot $J$ (regarded as a knot in $S^3$ as in \Cref{fig:I-bundle}) forms a basis for $H_1(\partial N)$.  As a consequence of the computations from \Cref{H1(N)}, $H_1(S_\phi)$ has matrix presentation $$\left[\begin{array}{rrrr}
0 & 0 & -2 a & -2 c \\
2 & 0 & d & b \\
0 & 0 & -2 & 0 \\
0 & 2 & 0 & 1
\end{array}\right] $$

If $d\equiv 0$, this matrix has Smith normal form given by the $4\times 4$ diagonal matrix with diagonal $(1,2,2,4c)$. Otherwise, the Smith normal form has diagonal $(1,1,4,4c)$. 
\end{proof}

\begin{remark}\label{rmk:c>0}
The above argument shows that $|H_1(S_\phi)| = 16|c|$.  It is natural to ask if the sign of $c$ affects the Sol manifold.  Note that $N$ admits an orientation-preserving diffeomorphism that acts by $-1$ on $\partial N$, as seen from the strong inversion on $J$ in \Cref{fig:I-bundle}.  Therefore, any Sol manifold is orientation-preserving diffeomorphic to some $S_\phi$ with $c \geq 0$. (For more details, see \cite{bgw}.)    
\end{remark}

\begin{lemma}\label{woop} The Sol manifolds $S_\phi$ and $S_{\phi^{-1}}$ are orientation-preserving diffeomorphic.  
\end{lemma}
\begin{proof}By definition, $S_\phi=N_1\cup_\phi N_2$ with the convention that $\phi:\partial N_2\to \partial N_1$. Switching the roles of the two copies of $N$ we get $N_2\cup_{\phi^{-1}} N_1$, with the convention that the domain of the gluing map is part of the rightmost manifold. The latter is clearly $S_{\phi^{-1}}$ showing that $S_\phi=S_{\phi^{-1}}$.
\end{proof}

The following lemma states gives the first explicit relationship between Sol and dihedral manifolds that we need.

\begin{lemma}\label{lem:sol-splice}
The Sol manifold $S_\phi$ can be described as a generalized splice of $D_{-b/c}$ and $D_{a/c}$ along the singular fibers of order $b$ and $a$ as in \Cref{fig:plumbing}.
\end{lemma}
\begin{proof}
Notice that the function $\phi$ identifies the curves $\lambda_1$ with $c\mu_2-a\lambda_2$, and $c\mu_1+b\lambda_1$ with $\lambda_2$. Thus, to get the desired splice it is enough to find curves $\gamma_1$ in $D_{-b/c}$ and $\gamma_2$ in $D_{a/c}$ with exterior $N_i$ such that $\mu_{\gamma_1}=c\mu_1+b\lambda_1$, $\mu_{\gamma_2}=c\mu_2-a\lambda_2$, and the rational longitudes $\lambda_{\gamma_i}$ equal $\lambda_i$.    

In the surgery description for $D_{-b/c}$ corresponding to the structure $ S^2\left (0;(2,1),(2,-1) ,(b,-c)\right )$ as in \Cref{fig:dihedral}, replace the curve with framing $-b/c$ by a chain of linked unknots corresponding to the continued fraction expansion $-b/c=[B;x_1,\ldots, x_n]^-$, so that the component with framing $B$ is a meridian of $J_1$. Let $\gamma_1$ be a meridian for the last circle in the chain. The slam-dunks give an identification of the torus $\partial N(\gamma_1)$ with $\partial N(J_1)$ via the map

\begin{equation}
f_1=\begin{bmatrix}0&-1\\1&0\end{bmatrix}\begin{bmatrix}B&-1\\1&0\end{bmatrix}\begin{bmatrix}x_{1}&-1\\1&0\end{bmatrix}\cdots \begin{bmatrix}x_{n}&-1\\1&0\end{bmatrix}=\begin{bmatrix}0&-1\\1&0\end{bmatrix}\begin{bmatrix}-b&c\\c'&b'\end{bmatrix}.
\end{equation}

Next, in the identification $D_{a/c}=S^2\left (0;(2,1),(2,-1) ,(a,c)\right )$, replace the framing $a/c$ by an integer $A$ defined by the equality $a/c=[A;x_n,\ldots, x_1]^-$, or equivalently, by the equation $A=ac'-b'd$. Let $\gamma_2$ be a meridian for that curve. In this case the isomorphism $f_2:\partial N(\gamma_2)\to\partial N(J_2)$ is given by 

\begin{equation}
f_2=\begin{bmatrix}0&-1\\1&0\end{bmatrix}\begin{bmatrix}A&-1\\1&0\end{bmatrix}=\begin{bmatrix}-1&0\\A&-1\end{bmatrix}
\end{equation}

This realizes the respective complements of $\gamma_1,\gamma_2$ in $D_{-b/c},D_{a/c}$ as the result of surgery on a solid torus. Since the composition $f_1^{-1}\circ\phi\circ f_2$ equals $\left(\begin{smallmatrix}0 & 1 \\1 & 0\end{smallmatrix}\right)$, the manifold $S_\phi$ is the actual splice of $D_{-b/c}\setminus N(\gamma_1)$, and $D_{a/c}\setminus N(\gamma_2)$. From this we obtain a description of $S_\phi$ as in \Cref{fig:plumbing}.

\end{proof}

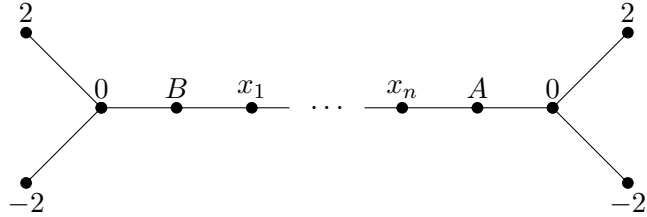
\begin{figure}[h]
\centering
\begin{tikzpicture}
  \filldraw (0,1) circle (2pt) node [align=center, above] {$2$} --(1,0) circle (2pt)--(0,-1) circle (2pt) node [align=center, below] {$-2$};
   \filldraw(8,1) circle (2pt) node [align=center, above] {$2$}--(7,0) circle (2pt)--(8,-1) circle (2pt) node [align=center, below] {$-2$};
  \filldraw (1,0) circle (2pt)node [align=center, above] {$0$}--(2,0) circle (2pt)node [align=center, above] {$B$}-- (3,0) circle (2pt)node [align=center, above] {$x_1$}--(3.5,0);
  \filldraw (4.5,0)--(5,0) circle (2pt) node [align=center, above] {$x_n$}-- (6,0) circle (2pt)node [align=center, above] {$A$}--(7,0) circle (2pt) node [align=center, above] {$0$};
  \draw (4,0) node {$\ldots$};
  \end{tikzpicture}
\caption{A general Sol manifold as a splice of dihedral manifolds.}\label{fig:plumbing}
\end{figure}

The lemma above considers the general case of dihedral manifolds. In \Cref{spinc_sol} and \Cref{main_proof} we will focus on the case $c=1$, and this case has the following more succinct description: if $\gamma_1$ is a meridian for the singular fiber of order $-b$ in $D_{-b}$, then there is a homeomorphism between $D_{-b}\setminus N(\gamma_1)$ and $N_1=\R P^3\# \overline{\R P^3}\setminus N(J_1)$ that identifies the rational longitudes, and that identifies a meridian of $\gamma_1$ with the $b$-framing of $J$. Analogously, taking $\gamma_2$ to be the meridian of the singular fiber of order $a$ in $D_a$ gives an identification $(D_a \setminus N(\gamma_2),\lambda_{\gamma_2},\mu_{\gamma_2})$ with $(N_2, \lambda_{J_2},\mu_{J_2}-a\lambda_{J_2})$. So, the manifold $S_\phi$ corresponds precisely to the gluing of $D_{-b}\setminus N(\gamma_1)$ and $D_{a}\setminus N(\gamma_2)$ along their boundaries determined by the identification $\lambda_{\gamma_2}\to \mu_{\gamma_1}$, $\mu_{\gamma_2}\to \lambda_{\gamma_1}$.  Since $\lambda_{\gamma_i}$ and $\mu_{\gamma_i}$ are distance 1 in this case, this is an honest splice, with diagram given in \Cref{fig:plumbingc=1}.

\begin{figure}[h]
\centering
\def\svgwidth{0.5\textwidth}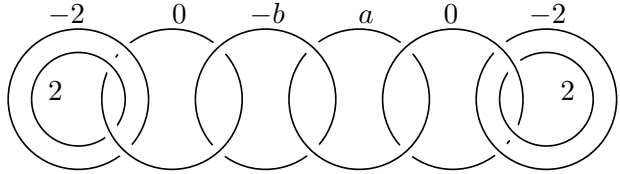
\caption{A Sol manifold as a splice of dihedral manifolds when $c = 1$.}\label{fig:plumbingc=1}
\end{figure}

\subsection{Rational homology cobordisms between Sol and dihedral manifolds}\label{sec:rational-sol-dihedral}

\begin{figure}
\centering
\def\svgwidth{0.125\textwidth}\input{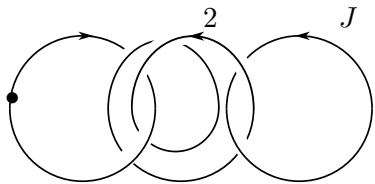}
\caption{The result of attaching a 1-handle and a 2-handle to the cylinder over the unknot complement $S^3\setminus N(J)$.}\label{fig:rel_cob}
\end{figure}

\begin{proposition}\label{prop:cobs} There exist rational homology cobordisms $W_{a/c}, W_{-b/c}$ from $D_{a/c},D_{-b/c}$ (respectively) to $S_\phi$. The homology groups admit isomorphisms 
$$
H_1\left(W_{a/c}\right)\cong \Z/2\oplus \Z/4c \quad\quad H_1\left(W_{-b/c}\right)\cong\begin{cases}
\Z/2\oplus \Z/4c & \text{if } b\equiv 1 \pmod{2}\\
\Z/4\oplus \Z/2c & \text{else.}
\end{cases}
$$
Moreover, the maps on $H_1$ induced by the inclusions $\partial W\to W$ can be computed explicitly as in \Cref{maps_H1}.
\end{proposition}

\begin{remark}
\Cref{prop:cobs} gives the appearance of a discrepancy between $H_1(W_{a/c})$ and $H_1(W_{-b/c})$.  This is not the case, since when $b \equiv 0 \pmod{2}$, $c$ is odd, and so $\Z/4 \oplus \Z/2c \cong \Z/2 \oplus \Z/4c$.  We have chosen to express the groups in this way as it makes the maps in \Cref{maps_H1} easier to work with.   
\end{remark}

\begin{proof} Let $J\subset S^3$ be an unknot, and construct a 4-manifold $X$ as  \begin{equation}\label{eq:rel_cob}X=[0,1]\times \left(S^3\setminus N(J)\right)\bigcup \left( 1-\text{handle}\cup 2-\text{handle}\right),\end{equation} with attaching information as in \Cref{fig:rel_cob}. This 4-manifold is a relative cobordism from the solid torus $S^3\setminus N(J)=D^2\times S^1$ to $N$, the twisted $I$-bundle over the Klein bottle. From the construction it is clear that the so-called incoming `vertical' boundary component of $X$ is $S^3\setminus N(J)$. To see that the outgoing boundary component of $X$ is $N$, notice that sliding the $-2$-framed handle over the $2$-framed handle in \Cref{fig:I-bundle} produces a Kirby diagram like the one in \Cref{fig:rel_cob}, modulo replacing the dotted unknot representing the 1-handle with a $0$-framed 2-handle. Notice also that $\partial_h X$ the `horizontal' boundary component of $X$ is diffeomorphic to $[0,1]\times T^2$, and the identification with each $T^2$ is always given by the choice of meridian-longitude pair of the knot $J$, even when considered as a knot in $\R P^3\# \overline{\R P^3}$. In terms of homology, notice that if $z$ represents the homology class determined by the oriented meridian of the dotted circle, then 
\begin{equation}\label{H1(X)}
H_1(X)=\langle \mu,\lambda, z \mid \lambda=0,\; 2z+\mu=0\rangle=\langle z\rangle= \Z
\end{equation}

Let $X_1,X_2$ be two copies of the 4-manifold $X$, and construct 
\begin{equation}\label{eq:cob}
W_{-b/c}= \left([0,1]\times N_1\right)\cup_\Phi X_2, \quad\text{and}\quad W_{a/c}=X_1\cup_\Phi \left([0,1]\times N_2\right),\end{equation} 
where by $\Phi$ we mean the self-diffeomorphism $\phi$ for each fixed value of $t\in[0,1]$. The fact that $X_i$ is a relative cobordism from a solid torus to $N_i$ implies that $W_{a/c}$, $W_{-b/c}$ are, respectively, cobordisms from $D_{a/c}$ and $D_{-b/c}$, to $S_\phi$.  

To get the maps on homology, it will be beneficial to express each manifold $M$ involved (be it 3- or 4-manifolds) as a decomposition $M=M_1\cup_{F} M_2$ where $F$ is separating. This is because in this case the Mayer-Vietoris sequence will have the general form $$H_1(F)\to H_1(M_1)\oplus H_1(M_2)\to H_1(M),$$ thus realizing $H_1(M)$ as the cokernel of a map of the form $x\to (g_1(x),g_2(x))$. Choosing $F=\partial M_2$ allows us to take $g_2$ to be the map on homology induced by the inclusion $\iota: \partial M_2\to M_2$, and $g_1$ induced by the composition $\iota\circ\phi:\partial M_2\to M_1$. When $H_1(F)$ is torsion free (which is the present case), if $P_i$ denotes a matrix presentation for $H_1(M_i)$ and $G_i$ one for $g_i$, the group $H_1(M)$ has matrix presentation $$\left[\begin{array}{c|c|c}P_1& 0 & G_1 \\\hline 0& P_2 & G_2\end{array}\right].$$ 
The specific decompositions that we consider are 
\arraycolsep=0.5ex\def\arraystretch{0.75}
{\small
\begin{align*}
& D_{-b/c}=N_1\underset{\phi}{\cup} D^2\times S^1, &&D_{a/c}=D^2\times S^1 \underset{\phi}{\cup} N_2, &&W_{-b/c}=\left([0,1]\times N_1\right)\underset{\Phi}{\cup} X_2, &&W_{a/c}=X_1\underset{\Phi}{\cup} \left([0,1]\times N_2\right),
\intertext{and these induce the following matrix presentations for the first homology groups:}
&\left[\begin{array}{rr|rr}
0 & 0 & -2 a & -2 c \\
2 & 0 & d & b \\\hline
0 & 0 & 1 & 0 \\
0 & 1 & 0 & 1
\end{array}\right],&
&\left[\begin{array}{rr|rr}
0 & 0 & a & c \\
1 & 0 & d & b \\\hline
0 & 0 & -2 & 0 \\
0 & 2 & 0 & 1
\end{array}\right],&
&
\left[\begin{array}{rr|rr}
0 & 0 & -2 a & -2 c \\
2 & 0 & d & b \\\hline
0 & 0 & -2 & 0 \\
0 & 1 & 0 & 0
\end{array}\right],&
&
\left[\begin{array}{rr|rr}
0 & 0 & -2 a & -2 c \\
1 & 0 & d & b \\\hline
0 & 0 & -2 & 0 \\
0 & 2 & 0 & 1
\end{array}\right].
\end{align*} }Here we are using the fact that the maps on $H_1$ induced by the inclusions $\partial N\hookrightarrow N$, $\partial D^2\times S^1\hookrightarrow D^2\times S^1$, $\partial_h X\simeq \partial D^2\times S^1\hookrightarrow X$ have, respectively, matrix presentation $\left(\begin{smallmatrix}-2 & 0 \\\phantom{-}0 & 1\end{smallmatrix}\right)$, $\left(\begin{smallmatrix}1 & 0 \\0 & 1\end{smallmatrix}\right)$, and $\left(\begin{smallmatrix}-2 & 0 \\\phantom{-}0 & 0\end{smallmatrix}\right)$ with respect to the bases $\{\mu,\lambda\}$ for $H_1(\partial N)=H_1(\partial D^2\times S^1)$, $\{\mu,\lambda,y\}$ for $H_1(N)$, and $\{\mu,\lambda,z\}$ for $H_1(X)$. Similar to the proof of \Cref{prop:H1(sol)}, the Smith normal from of the matrices depends on the parities of the parameters. Straightforward computations give the claimed structure for $H_1(W_{-b/c})$ and $H_1(W_{a/c})$. \\

To get the maps induced by the inclusions of the 3-manifolds into the cobordisms $W_{-b/c},W_{a/c}$, we focus on the inclusions of the constituent pieces of the 3-manifolds into those of the cobordisms. Namely, we focus on the maps $H_1(N)\to H_1(X)$ and $H_1(D^2\times S^1)\to H_1(X)$ (the others are given by the identity). With respect to our chosen generators, these maps are given respectively by, $y\to z, \lambda\to 0$, and $\mu\to -2z, \lambda\to 0$. Since our identifications between the relevant first homology groups and finite abelian groups rely on the computation of Smith normal forms, we need to take those into consideration in our computations of the maps induced by inclusions. Our strategy is straight forward and is as follows: let $P$ be a presentation matrix, and $F,C$ unimodular matrices such that $D=FPC$, where $D$ is the Smith normal form of $P$. Let $q$ be the quotient map $\Z^4\to \text{coker}(D)$. If $h:\Z^4\to \Z^4$ is such that $h(\Ima(P))\subseteq \Ima(P')$, then the induced map $\overline{h}: \text{coker}(D)\to \text{coker}(D')$ is given by $q'\circ F'\circ h\circ F^{-1} \circ q^T$. For example, if the parameters are such that $a,c,d$ are odd and $b$ is even, then $H_1(S_\phi)$ has Smith normal form
$$
\arraycolsep=0.5ex\def\arraystretch{0.75}
\small
D=\left(\begin{array}{rrrr}
1 & 0 & 0 & 0 \\
0 & 1 & 0 & 0 \\
0 & 0 & 4 & 0 \\
0 & 0 & 0 & 4 c
\end{array}\right)=\left(\begin{array}{rrrr}
0 & 0 & 0 & 1 \\
0 & 1 & 0 & 0 \\
0 & 2 & 1 & 0 \\
1 & 0 & -a & 2 c
\end{array}\right)
\left[\begin{array}{rrrr}
0 & 0 & -2 a & -2 c \\
2 & 0 & d & b \\
0 & 0 & -2 & 0 \\
0 & 2 & 0 & 1
\end{array}\right] 
\left(\begin{array}{ccrr}
-\tfrac{b}{2} & \tfrac{1-d}{2} & d & b \\
0 & 0 & 0 & 1 \\
0 & 1 & -2 & 0 \\
1 & 0 & 0 & -2
\end{array}\right)=FPC,$$
and the cobordism $W_{-b/c}$ has first homology with Smith normal form given by 
$$
\arraycolsep=0.5ex\def\arraystretch{0.75}
\small
D'=\left(\begin{array}{rrrr}
1 & 0 & 0 & 0 \\
0 & 1 & 0 & 0 \\
0 & 0 & 4 & 0 \\
0 & 0 & 0 & -2 c
\end{array}\right)=
\left(\begin{array}{rrrr}
0 & 0 & 0 & 1 \\
0 & 1 & 0 & 0 \\
0 & 2 & 1 & 0 \\
1 & 0 & -a & 0
\end{array}\right)
\left[\begin{array}{rrrr}
0 & 0 & -2 a & -2 c \\
2 & 0 & d & b \\
0 & 0 & -2 & 0 \\
0 & 1 & 0 & 0
\end{array}\right]
\left(\begin{array}{rrrr}
0 & -\frac{1}{2} d + \frac{1}{2} & d & -\frac{1}{2} b \\
1 & 0 & 0 & 0 \\
0 & 1 & -2 & 0 \\
0 & 0 & 0 & 1
\end{array}\right)=F'P'C'.$$
With respect to the meridian-longitude generators, the map on $H_1$ induced by the inclusion $S_\phi\to W_{-b/c}$ has matrix representation a diagonal matrix $\Delta$ with diagonal $(1,1,1,0)$. Thus, the matrix form of $H_1(S_\phi)\to H_1(W_{-b/c})$ with respect to the bases that give a diagonal presentation for $H_1(S_\phi), H_1(W_{-b/c})$ is the lower $2\times 2$ diagonal block of the matrix $$F'\circ \Delta\circ F^{-1}=\left(\begin{array}{rrrr}
0 & 0 & 0 & 0 \\
0 & 1 & 0 & 0 \\
0 & 0 & 1 & 0 \\
-2 c & 0 & 0 & 1
\end{array}\right).$$
 The remaining maps on first homology induced by inclusion can be computed directly as in \Cref{maps_H1}.  
\begin{table}

\resizebox{0.75\textwidth}{!} {
\begin{tabular}{|>{$}c<{$}|>{$}c<{$}|>{$}c<{$}|>{$}c<{$}|>{$}c<{$}|>{$}c<{$}|>{$}c<{$}|>{$}c<{$}|}\hline
&&&&&&&\\
a&b&c&d&S_\phi
\to W_{-b/c}
&D_{-b/c}
\to W_{-b/c}
&S_\phi
\to W_{a/c}
&D_{a/c}
\to W_{a/c}
\\[1ex]\hline
0&0&1&1& \left(\begin{array}{rr}
1 & 0 \\
0 & 1
\end{array}\right) 
& \left(\begin{array}{rr}
2 & 0 \\
0 & 1
\end{array}\right)
& \left(\begin{array}{rr}
1 & 0 \\
0 & 1
\end{array}\right)
& \left(\begin{array}{rr}
1 & 0 \\
0 & -2
\end{array}\right)\\
0&1&1&1& \left(\begin{array}{rr}
1 & 0 \\
0 & 1
\end{array}\right) 
& \left(\begin{array}{r}
0 \\
1
\end{array}\right) 
& \left(\begin{array}{rr}
1 & 0 \\
c & 1
\end{array}\right)
& \left(\begin{array}{rr}
1 & 0 \\
0 & -2
\end{array}\right)\\
1&0&1&1& \left(\begin{array}{rr}
1 & 0 \\
0 & 1
\end{array}\right)
& \left(\begin{array}{rr}
2 & 0 \\
0 & 1
\end{array}\right) 
& \left(\begin{array}{rr}
1 & 0 \\
0 & 1
\end{array}\right) 
& \left(\begin{array}{r}
1 \\
-1
\end{array}\right)\\
1&1&1&0& \left(\begin{array}{rrr}
0 & 1 & 0 \\
0 & 0 & 1
\end{array}\right) 
& \left(\begin{array}{r}
0 \\
1
\end{array}\right) 
& \left(\begin{array}{rrr}
0 & 1 & 0 \\
2 c & 0 & 1
\end{array}\right)
& \left(\begin{array}{r}
1 \\
-1
\end{array}\right)\\
1&1&0&1& \left(\begin{array}{rr}
1 & 0 \\
0 & 1
\end{array}\right)
&\left(\begin{array}{r}
0 \\
1
\end{array}\right)
& \left(\begin{array}{rr}
1 & 0 \\
c & 1
\end{array}\right) 
& \left(\begin{array}{r}
1 \\
-1
\end{array}\right)
\\
1&1&0&0& \left(\begin{array}{rrr}
0 & 1 & 0 \\
0 & 0 & 1
\end{array}\right) 
& \left(\begin{array}{r}
0 \\
1
\end{array}\right)
& \left(\begin{array}{rrr}
0 & 1 & 0 \\
2 c & 0 & 1
\end{array}\right)
& \left(\begin{array}{r}
1 \\
-1
\end{array}\right)
\\\hline
\end{tabular}}
\medskip
\caption{Maps on first homology induced by the inclusions of dihedral/Sol manifolds into the rational cobordisms relating them.}\label{maps_H1}
\end{table}

\end{proof}

The above arguments in fact show: 
\begin{corollary}\label{sol-ball} 
Every Sol rational homology sphere with $|H_1| = 16$ bounds a rational homology ball.  More generally, $S_\phi$ is rationally homology cobordant to an order $|c|$ lens space.  
\end{corollary}
\begin{proof}
Following the notation of the above proof, $X_1 \cup_\Phi X_2$ is a rational homology cobordism from an order $|c|$ lens space to $S_\phi$.  
\end{proof}

\section{Spin$^c$ structures}\label{sec:spinc}
In this section, we provide some background on spin$^c$ structures that will help us compare the values of the $d$-invariants of potentially homology cobordant Sol manifolds. Recall that all compact, smooth, orientable three- and four-manifolds admit spin$^c$ structures.  For $M$ a smooth manifold admitting a spin$^c$ structure, a choice of $\mathfrak{s}_0\in Spin^c(M)$ determines a (non-canonical) bijection 
\begin{equation}\label{spinc:bijection}
H^2(M)\to Spin^c(M)
\end{equation} given by $\alpha\to \mathfrak{s}_0+\alpha$. Alternatively, the difference of two elements of $\mfs,\mfs' \in Spin^c(M)$ is a well-defined element $\mfs - \mfs' \in H^2(M)$. Furthermore, there is a canonical function \begin{equation} c_1: Spin^c(M)\to H^2(M)\end{equation} (which may be neither injective nor surjective) such that:
\begin{itemize}
\item $c_1(\mathfrak{s}+\alpha)=c_1(\mathfrak{s})+2\alpha$.
\item Conjugation induces an involution $Spin^c(M)\to Spin^c(M)$ that reverses signs of Chern classes and the $H^2(M; \Z)$-action. That is, $\overline{\mfs + \alpha} = \overline{\mfs} - \alpha$ and 
$c_1(\overline{\mathfrak{s}})=-c_1(\mathfrak{s})$.  
\item The set of $Spin^c$ structures with first Chern class 0 
is precisely the set of self-conjugate elements. 
\item On a closed, oriented, three-manifold, $\text{Im}\left(c_1\right)$ is exactly the set $2H^2(M)=\{2\alpha \mid \alpha \in H^2(M)\}$.  
\end{itemize}
For more details about spin$^c$ structures, we refer the reader to \cite[Proposition 1]{gompf-spinc}.  

Spin$^c$ structures also behave well with respect to cobordisms.  If $\mfs$ is a spin$^c$ structure on a cobordism $W: M_1 \to M_2$, then there are restrictions $\mfs|_{M_i}$.  Furthermore, the restrictions respect the affine $H^2$ structure: 
\begin{equation}
(\mfs + \alpha)|_{M_i} = \mfs|_{M_i} + \alpha|_{M_i}, \ \alpha \in H^2(W).
\end{equation}
 
Usually, given a cobordism $W: M_1 \to M_2$, to understand the restriction maps on spin$^c$ structures, it is actually easier to compute the induced maps on $H_1$ and then convert to $H^2$.  We now describe this more carefully.  Let $W: M_1 \to M_2$ be a rational homology cobordism between rational homology three-spheres.  Let $\iota_i$ be the maps on $H_1$ induced by the inclusions of $M_i$ into $W$.  Then $H^2(W)$ is given by $Ext(H_1(W), \mathbb{Z})$ and similarly for the $M_i$; the restriction of $\alpha$ to $M_i$ is then given by $\iota_i^*(\alpha)$, where $\iota_i$ is the induced map on $Ext$.  (Note that there is no Poincar\'e duality needed here.)   For example if $\iota _i: \Z/2 \to \Z/4$ is multiplication by 2, then the induced map on $Ext$ from $\Z/4$ to $\Z/2$ is the quotient map; conversely, if the map from $\Z/4$ to $\Z/2$ is the quotient map, then the map on $Ext$ is multiplication by 2.  \Cref{maps_H1} shows that the maps on $H_1$ that we consider are always given component-wise by one of these maps, the identity map, or zero. The corresponding homomorphisms on $H^2$ are included in \Cref{diagrams}. (By \Cref{woop}, we can ignore the case $a \equiv 1, b \equiv 0 \pmod{2}$.)  Finally, to get the induced maps on spin$^c$ structures it is enough to pick a base spin$^c$ structure on $W$, say $\mfs$, and consider its restriction $\mfs_i$ to $M_i$ to get 
\begin{equation}
(\mfs + \alpha)|_{M_i} = \mfs_i + \iota_i^*(\alpha). 
\end{equation}

\subsection{$d$-invariants}
For a rational homology three-sphere, Heegaard Floer homology assigns a function $d: Spin^c(M) \to \mathbb{Q}$, called the {\em d-invariants} or {\em correction terms}.  (See \cite{OS} for the definition and properties.)  The $d$-invariants have two useful properties which we will exploit.  If $\mfs$ and $\mfsbar$ are conjugate spin$^c$ structures on $M$, then $d(M, \mfs) = d(M, \mfsbar)$.  Furthermore, if $W: M_1 \to M_2$ is a rational homology cobordism, then for any spin$^c$ structure $\mfs$ on $W$, we have $d(M_1, \mfs|_{M_1}) = d(M_2, \mfs|_{M_2})$. \\ 

An important component of our work is the connection between Heegaard Floer homology and the Casson-Walker-Lescop invariants.  If $Y$ is an L-space, by \cite[Theorem 3.3]{rustamov} we have 
\begin{equation}\label{d=casson}
\sum_{\mathfrak{t} \in Spin^c(Y)} d(Y,\mathfrak{t}) = -2\lambda_L(Y), 
\end{equation}
where $\lambda_L$ denotes the Lescop invariant, normalized so that if $\Sigma(2,3,5)$ is the boundary of the negative definite $E_8$ manifold, then $\lambda_L\left(\Sigma(2,3,5)\right)=-1$. Recall that for $Y$ a rational homology sphere, $\lambda_L(Y) = |H_1(Y)| \lambda_W(Y)$ where $\lambda_W$ is the Casson-Walker invariant. When $Y$ is a dihedral manifold of the form $D_n$ for $n\in\Z$, Doig computes in \cite[Example 15]{doig} the four Heegaard Floer $d$-invariants of $D_n$ as\footnote{Our sign conventions are chosen so as to agree with those of \cite[Corollary 1.5]{OS-plumbed} and of Lescop's invariant.}: 
\begin{align}\label{eq:d-dih}
d(D_n) &= \left\{0, 0, \tfrac{n+2}{4}, \tfrac{n-2}{4}\right\}
\end{align}
From \Cref{eq:d-dih} and \Cref{d=casson} we get a simple formula for the Lescop invariant for order 4 dihedral manifolds as follows: 
\begin{equation}\label{eq:dihed}
\lambda_L(D_n) = -\frac{1}{2}\sum_{\mathfrak{t} \in Spin^c(D_n)} d(D_n,\mathfrak{t}) =-\frac{n}{4}.
\end{equation}

\begin{remark}\label{dih-self} Recall that for any $\mathfrak{u}\in Spin^c(D_n)$ we have the equality $d(D_n,\mathfrak{u})=d(D_n,\overline{\mathfrak{u}})$. This shows that if $n$ is an odd integer, and $\mathfrak{u}\in Spin^c(D_n)$, then $ \mathfrak{u}\neq \overline{\mathfrak{u}}$ is equivalent to $d(D_n,\mathfrak{u})=0.$ 
\end{remark}

\section{Spin$^c$ structures on Sol manifolds}\label{spinc_sol}
This section presents a detailed analysis of the spin$^c$ structures on Sol manifolds and a computation of (most of) the $d$-invariants.  We will specialize to the case of $c = 1$ throughout, which as discussed in \Cref{rmk:c>0} covers all oriented diffeomorphism classes of Sol rational homology spheres with order 16 first homology.  We set the following notation: For $a, b \in \mathbb{Z}$,  $\Sol$ denotes the Sol rational homology sphere with gluing matrix $\phi = \begin{pmatrix} a & 1 \\ ab+1 & b \end{pmatrix}$.  
As discussed earlier, if $a$, $b$, or $ab+1$ is 0, then $\Sol$ is Seifert fibered instead of Sol.  For notational simplicity, we treat these as degenerate Sol manifolds.  Our arguments will involved a case analysis based on the parities of $a$ and $b$.  By \Cref{woop}, we can always avoid the case of $a \equiv 1, b \equiv 0 \pmod{2}$, so we will not need to mention this case in what follows. 

We write $\wb: H^2(W_{-b}) \to H^2(\Sol)$, $\wa: H^2(W_a) \to H^2(\Sol)$, $\db: H^2(W_{-b}) \to H^2(D_{-b})$, and $\da: H^2(W_a) \to H^2(D_a)$ for the associated restriction maps.  We also let $\solextend$ denote the spin$^c$ structures on $\Sol$ that extend over both $W_{-b}$ and $W_a$, $\solextendb$ denote the spin$^c$ structures on $\Sol$ that extend over $W_{-b}$ and not $W_a$, and similarly for $\solextenda$.  Finally, $\solnoextend$ denotes the spin$^c$ structures on $\Sol$ that neither extend over $W_{-b}$ nor $W_a$.  In what follows, we will rely heavily on the second cohomology computations shown in \Cref{diagrams}.  We remark that the identifications between the second cohomology groups and the specific abelian groups in \Cref{prop:cobs} are non-canonical.  We will also use the properties of spin$^c$ structures discussed in \Cref{sec:spinc} without explicit reference.

\begin{lemma}\label{lem:extend} For the Sol manifold $\Sol$ and rational homology cobordisms $W_{-b},W_{a}$ from \Cref{prop:cobs}, let $\solextend$ be the set of elements of $Spin^c(\Sol)$ that extend to both $W_{-b}, W_a$. Then $\solextend$ is non-empty, and, moreover, for fixed $\theta \in \solextend$:   
\begin{enumerate}
\item the associated bijection $H^2(\Sol)\to Spin^c(\Sol)$ can be partitioned into four bijections of 4-element sets: 
\begin{align*}
\Ima(\wb) \cap \Ima(\wa)  &\to \solextend, \\
\Ima(\wb) \setminus \Ima(\wa)  &\to \solextendb, \\  
\Ima(\wa) \setminus \Ima(\wb)  &\to \solextenda, \\
H^2(\Sol) \setminus \left(\Ima(\wb) \cup \Ima(\wa)\right) &\to \solnoextend.
\end{align*}
\item The values of $c_1$ on $\solextend, \solextenda, \solextendb, \solnoextend$ are as given in \Cref{chern-classes} following the identifications in \Cref{prop:cobs}.
\end{enumerate}
\end{lemma}

\begin{proof} To see that $\solextend$ is non-empty, consider the smooth 4-manifold $Z=W_{-b}\cup_{id}\overline{W_a}$. Then $Z$ admits a spin$^c$ structure $\mfs$, and so $\theta=\mfs|_{\Sol}$ is an element of $\solextend$.  We fix this $\theta$ for the remainder of the proof.  

Let $\theta' \in Spin^c(\Sol)$ extend to $\mathfrak{s_b}\in Spin^c(W_{-b})$.  Then $\theta' + \alpha$ extends over $W_{-b}$ if and only if $\alpha \in \Ima(\wb)$.  Indeed, if $\alpha = \wb(\beta)$, then $\theta' + \alpha$ is the restriction of $\mathfrak{s_b} + \beta \in Spin^c(W_{-b})$ to $\Sol$. Conversely, if $\theta' + \alpha$ admits an extension $\mathfrak{t_b}\in Spin^c(W_{-b})$, then $\alpha = (\theta' + \alpha) - \theta'=\wb(\mathfrak{s_b} - \mathfrak{t_b})$. The analogous statements hold for extensions over $W_a$, and this gives the claimed bijections. The fact that each set has exactly four elements can be deduced directly from the maps in \Cref{diagrams}.  This establishes the first item.  

For the second item, we begin by showing that the first Chern class of the spin$^c$ structures are constant on the sets $\solextend, \solextenda$, etc.  We begin with $\solextend$.  First, suppose that $\theta' \in \solextend$.  Applying the homomorphism $c_1$ to the equality $\theta'=\theta+(\theta'-\theta)$ gives $$c_1(\theta')=c_1(\theta)+2(\theta'-\theta),$$ 
where by the discussion above, $\theta'-\theta \in \Ima(\wb) \cap \Ima(\wa)$. Direct computation using \Cref{diagrams} shows that $2\cdot\left(\Ima(\wb) \cap \Ima(\wa)\right)=0$, and so $c_1(\theta')=c_1(\theta)$.  The other cases work similarly, except the values of $c_1$ may not agree with $c_1(\theta)$.  The key observation is that each of $$2\left(\Ima(\wb) \setminus \Ima(\wa)\right),\: 2\left(\Ima(\wa) \setminus \Ima(\wb))\right),2\left(H^2(\Sol) \setminus \Ima(\wa) \cup \Ima(\wb))\right),$$ always consists of a single element, and this element is precisely the difference of the value of $c_1(\theta)$ and the first Chern class for the elements of $\solextendb$, $\solextenda$ or $\solnoextend$ respectively.  In conclusion, we only need to compute $c_1$ for a single element of any of the relevant sets of spin$^c$ structures.   

Now, we compute the first Chern class of the spin$^c$ structures by a case analysis on the parities of $a, b$.  First, suppose that $a, b \equiv 0 \pmod{2}$.  Note that in this case all spin$^c$ structures on $D_{-b}$ and $D_a$ have $c_1 = 0$.  From \Cref{diagrams}, we then see that if a spin$^c$ structure on $\Sol$ extends over $W_{-b}$, the second coordinate of $c_1$ must be 0.  Indeed, if $c_1(\theta') = (z,2)$ for some $z$ and $\theta'$ extends over $W_{-b}$, the first Chern class of the extension is of the form $(q,1)$ or $(q,3)$ in $H^2(W_{-b})$ and cannot restrict to be 0 on $D_{-b}$.  A similar argument shows that a spin$^c$ structure on $\Sol$ that extends over $W_a$ must have first coordinate of $c_1$ equal to 0.  Therefore, $c_1(\solextend) = \{(0,0)\}$ in $H^2(\Sol)$.  Next, we determine $c_1$ for $\solextendb$.  Note that $(1,0) \in \Ima(\wb) \setminus \Ima(\wa)$.  Therefore $\theta' = \theta + (1,0) \in \solextendb$.  We have $c_1(\theta') = (2,0)$.  This shows $c_1(\solextendb) = \{(2,0)\}$.  A similar argument shows $\theta + (0,1) \in \solextenda$ and so $c_1(\solextenda) = \{(0,2)\}$.  Finally, $\theta + (1,1) \in \solnoextend$ and $c_1(\solnoextend) = \{(2,2)\}$.  This completes the computation in the case $a, b \equiv 0 \pmod{2}$.  

Next, we consider $a \equiv 0, b \equiv 1 \pmod{2}$.  We begin by computing $c_1$ for $\theta \in \solextend$.  Let $\mfs_{-b}$ be an extension over $W_{-b}$ and $\mathfrak{u}_{-b}$ the restriction to $D_{-b}$; define $\mfs_a$ and $\mathfrak{u}_a$ similarly.  Then, $d(\Sol, \theta) = d(D_a, \mathfrak{u}_a) = d(D_{-b}, \mathfrak{u}_b)$.  Since the parities of $a, b$ are different, these $d$-invariants must be zero by \Cref{eq:d-dih}.  This means that $c_1(\mathfrak{u}_b) \neq 0$ by \Cref{dih-self}.  It follows from \Cref{diagrams} that $c_1(\theta)  = (z,2)$ for some $z$.   Because $c_1(\theta)$ is also an element of $\Ima(\wa)$, we must have $c_1(\mfs_a) = (0,2)$ or $(1,2)$; however, the latter implies $c_1(\mathfrak{u}_a) \neq 0$, which is not possible on $D_a$.  Therefore, $c_1(\mfs_a) = (0,2)$ and $c_1(\theta) = (2,2)$.  Note that $(0,1)$ (respectively $(1,1)$) is in $\Ima(\wb) \setminus \Ima(\wa)$ (respectively $\Ima(\wa) \setminus \Ima(\wb)$).  Therefore, $\theta + (0,1)$ (respectively $\theta + (1,1)$) is in $\solextendb$ (respectively $\solextenda$), and so we get 
$$c_1(\solextend) = \{(2,2)\}, \ c_1(\solextendb) = \{(2,0)\}, \ c_1(\solextenda) = \{(0,0)\}, \ c_1(\solnoextend) = \{(0,2)\}.$$        

Finally, we handle the case $a, b \equiv 1 \pmod{2}$.  There are only two possible values of $c_1$ in this case: $(0,0,0)$ and $(0,0,2)$.  Further, notice from \Cref{diagrams} that any spin$^c$ structure $\theta'$ on $\Sol$ has $c_1 = 0$ if and only if the restriction to $D_{-b}$ of an extension over $W_{-b}$ has $c_1 = 0$.  We begin with $\solextend$.  Let $x = (0,1,0) \in H^2(\Sol)$, which is in $\Ima(\wb) \cap \Ima(\wa)$, and so $\theta + x \in \solextend$ as well.  Since $x = \wb(1,0)$ and $\db(1,0) = 0$, we see that the unique extensions of $\theta$ and $\theta + x$ over $W_{-b}$ restrict to the same spin$^c$ structure on $D_{-b}$.  Therefore, $d(\Sol, \theta) = d(\Sol, \theta + x)$.  On the other hand, $x = \wa(1,0)$ and $\da(1,0) = 2$.  Therefore, the extensions of $\theta$ and $\theta + x$ over $W_a$ restrict to distinct spin$^c$ structures on $D_a$, which then necessarily have the same $d$-invariant.  Since $b$ is odd, this forces those two spin$^c$ structures to be the ones with $d = 0$ by \Cref{eq:d-dih}, i.e. the ones with $c_1 \neq 0$.  Therefore, $c_1(\theta) \neq 0$ and thus $c_1(\theta) = (0,0,2)$.  For $\solextendb$ (respectively $\solextenda$), we have $(0,0,1)$ (respectively $(1,0,1)$) is in $\Ima(\wb) \setminus \Ima(\wa)$ (respectively $\Ima(\wa) \setminus \Ima(\wb)$, and so 
$$c_1(\solextend) = \{(0,0,2)\}, \ c_1(\solextendb) = c_1(\solextenda) = \{(0,0,0)\}, \ c_1(\solnoextend) = \{(0,0,2)\}.$$        
This completes the claimed values in \Cref{chern-classes}.  
\end{proof}

\begin{table}
\begin{tabular}{|>{$}c<{$}|>{$}c<{$}|>{$}c<{$}|>{$}c<{$}|>{$}c<{$}|>{$}c<{$}|}\hline
a&b&\solextend&\solextendb&\solextenda&\solnoextend\\\hline
0&0&(0,0)&(2,0)&(0,2)&(2,2)\\\hline
0&1&(2,2)&(2,0)&(0,0)&(0,2)\\\hline
1&1&(0,0,2) &(0,0,0) & (0,0,0) &(0,0,2)\\\hline
\end{tabular}
\caption{The values of $c_1$ on the various spin$^c$ structures on $\Sol$.}\label{chern-classes}
\end{table}

Now that we have identified the spin$^c$ structures on $\Sol$, we can compute most of the $d$-invariants.  
\begin{proposition}\label{d-comp}
Let $a, b \in \Z$ and consider the Sol manifold $\Sol$.  Then, 
\begin{enumerate}
\item the $d$-invariants of $\solextend$ are $\{0, 0, 0, 0\}$;  
\item the $d$-invariants of $\solextendb$ are $\{\tfrac{-b+2}{4}, \tfrac{-b+2}{4},\tfrac{-b-2}{4},\tfrac{-b-2}{4} \}$; 
\item the $d$-invariants of $\solextenda$ are $\{\tfrac{a+2}{4}, \tfrac{a+2}{4},\tfrac{a-2}{4},\tfrac{a-2}{4} \}$.
\end{enumerate}
\end{proposition}

\begin{proof}
By \Cref{lem:extend} we can fix $\theta \in \solextend$.  
For the first item, notice that there exists $\alpha\in \Ima(\wb) \cap \Ima(\wa)\subseteq H^2(\Sol)$ such that $\alpha\in \wb\left(\ker(\db)\right)$ but $\alpha\notin \wa\left(\ker(\da)\right)$. Indeed, using the maps from \Cref{diagrams} we see that if $a\equiv 0\pmod{2}$, then $\alpha=(2,0)$, and if $a,b\equiv 1\pmod{2}$, then $\alpha=(0,1,0)$. Taking $\theta'=\theta+\alpha$ we get, 
\begin{align*}
d(\Sol,\theta)&=d(D_{-b},\mathfrak{u}_b)=d(\Sol,\theta')\\
d(\Sol,\theta)&=d(D_{a},\mathfrak{u}_a)\\
d(\Sol,\theta')&=d(D_{a},\mathfrak{u}_a+x)
\end{align*}
where $\mathfrak{u}_a, \mathfrak{u}_b$ are restrictions to $D_a, D_{-b}$ of an extension of $\theta$, and $x\neq 0\in H^2(D_a)$ is defined to be such that $x\in\da\left(\wa^{-1}(\alpha)\right)$. This forces $d(D_{a},\mathfrak{u}_a)=d(D_{a},\mathfrak{u}_a+x)$, and by \Cref{eq:d-dih}, this happens only if this value is exactly $0$.  This argument applies for any $\theta \in \solextend$ and so $d$ is identically 0 on $\solextend$.\footnote{One could also identify $\Ima(\wb) \cap \Ima(\wa)$ with $\wb\left(\ker(\db)\right) + \wa\left(\ker(\da)\right)$ and get the result directly.}

Next, we study $\solextendb$.  Note from \Cref{diagrams} that any spin$^c$ structure on $\Sol$ which has an extension over $W_b$ has a unique such extension, and the restriction map from spin$^c$ structures on $W_b$ to $D_{-b}$ is a surjective, 2-to-1 map.  Let $\theta'\in\solextendb$ and denote by $\mathfrak{v}_b$ the element of $Spin^c(D_{-b})$ obtained as the restriction of an extension of $\theta'$ to $W_{-b}$. As before, $d(\Sol,\theta')=d(D_{-b},\mathfrak{v}_b)$ but the above discussion shows that there are four spin$^c$ structures in $\solextend$ with $d = 0$ which correspond to two spin$^c$ structures on $D_{-b}$ with $d = 0$. This means that $d(D_{-b},\mathfrak{v}_b)\in\{\tfrac{-b + 2}{4}, \tfrac{-b - 2}{4}\}$, and in fact there are two choices for $\mathfrak{v}_b\in Spin^c(D_{-b})$. Since each choice gives two corresponding spin$^c$ structures on $\solextendb$ with the corresponding $d$-invariant, this gives the claimed four $d$-invariants for $\solextendb$.  (Note that it may be the case that one pair of these $d$-invariants could still be 0, but that only happens if $b=\pm 2$.) 

The case of $\solextenda$ can subsequently be deduced similarly.  
\end{proof}


\begin{table}
\resizebox{0.95\textwidth}{!}{
\begin{tabular}{|>{$}c<{$}|>{$}c<{$}|>{$}c<{$}|>{$}c<{$}|>{$}c<{$}|>{$}c<{$}|>{$}c<{$}|>{$}c<{$}|>{$}c<{$}|>{$}c<{$}|>{$}c<{$}|>{$}c<{$}|}\hline
a&b
&c_1(\mathfrak{u}_b)
&c_1(\mathfrak{u}_a)
&c_1(\mathfrak{s}_b)\in \db^{-1}(c_1(\mathfrak{u}_b))
&c_1(\mathfrak{s}_a)\in \da^{-1}(c_1(\mathfrak{u}_a))
&\wb\left(\db^{-1}(c_1(\mathfrak{u}_b))\right)
&\wa\left(\da^{-1}(c_1(\mathfrak{u}_a))\right)
&c_1(\theta)
\\\hline
0&0
&0&0&\{(0,0),(2,0)\}&\{(0,0),(0,2)\}
&\{(0,0),(2,0)\}
&\{(0,0),(0,2)\}
&(0,0)
\\\hline
0&1
&2&0&\{(1,2),(0,2)\}&\{(0,0),(2,2)\}
&\{(2,2),(0,2)\}
&\{(0,0),(2,2)\}
&(2,2)
\\\hline
1&1
&2&2&\{(0,2),(1,2)\}&\{(1,0),(0,2)\}
&\{(0,0,2),(0,1,2)\}
&\{(0,1,0),(0,0,2)\}
&(0,0,2)
\\\hline
\end{tabular}
}\medskip
\caption{Computations of the first Chern classes of the spin$^c$ structures obtained as extensions/restrictions of any element $\theta\in\solextend$.}\label{tab:extending-theta}
\end{table}

\section{Integer Homology Classification of Sol manifolds.}\label{main_proof}
In this section, we prove \Cref{thm:main}.  
We begin with the following: 
\begin{proposition}\label{d-sum-sol}
For $a, b \in \mathbb{Z}$, we have $$\sum_{\mfs \in Spin^c(\Sol)} d(\Sol, \mfs) = 2(a-b). $$
Consequently, if $\Sol$ and $\Solprime$ are homology cobordant then $a-b = a'-b'$, and further, $a \equiv a', b \equiv b' \pmod{2}$ or $a \equiv b', a' \equiv b \pmod{2}$. 
\end{proposition}
\begin{proof}
By \Cref{lem:sol-splice} and the splice-additivity of the Casson-Walker invariant (see for example \cite[Theorem 4.11]{saveliev-yellow}), we have that 
\begin{align*}
\lambda_W(\Sol) =& \lambda_W(D_{a}) + \lambda_W(D_{-b})
\intertext{and so using the fact that $\lambda_L(M) = |H_1(M)| \lambda_W(M)$, }
\frac{1}{16}\lambda_L(\Sol) =& \frac{1}{4}\lambda_L(D_{a}) + \frac{1}{4}\lambda_L(D_{-b}).
\intertext{This together with \eqref{eq:dihed} implies the desired result:}
\lambda_L(\Sol) =& 4\lambda_L(D_{a}) + 4\lambda_L(D_{-b})=b-a.
\end{align*}  
The first claim is then a direct application of \Cref{d=casson}, since Sol rational homology spheres are L-spaces \cite{bgw}.  The latter claim follows since the homology groups and the set of $d$-invariants are invariant under (integer) homology cobordism.
\end{proof}

We are now ready to complete the proof of \Cref{thm:main}.

\begin{proof}[Proof of \Cref{thm:main}]
Suppose that $\Sol$ and $\Solprime$ are homology cobordant.  Therefore, by possibly applying \Cref{d-sum-sol} and \Cref{woop}, we can assume $a \equiv a', b \equiv b' \pmod{2}$.  We now proceed with a case analysis based on the parities of the parameters.  

The first case we consider is that $a, a' \equiv 0, \ b, b' \equiv 1 \pmod{2}$ or vice versa.  
By possibly applying \Cref{woop}, we can assume that $a, a' \equiv 0, \ b, b' \equiv 1 \pmod{2}$.  Since self-conjugacy of spin$^c$ structures is preserved under homology cobordism, it is enough to compare the set of $d$-invariants of the self-conjugate spin$^c$ structures. By \Cref{lem:extend} the self-conjugate spin$^c$ structures on $\Sol$ are $\solextenda$.  By \Cref{d-comp}, the $d$-invariants of the self-conjugate spin$^c$ structures for $\Sol$ are therefore $\left\{\tfrac{a- 2}{4}, \tfrac{a- 2}{4}, \tfrac{a+2}{4}, \tfrac{a+2}{4}\right\}$.  We have the analogous result for $\Solprime$.  
The two sets of $d$-invariants are hence equal if and only if $a = a'$ and this forces $\Sol=\Solprime$ by \Cref{d-sum-sol}.

The next case is that all parameters $a, b, a', b'$ are odd.
Just like in the previous case, it is enough to compare the set of $d$-invariants of the self-conjugate spin$^c$ structures. By \Cref{lem:extend}, the self-conjugate spin$^c$ structures on $\Sol$ come from $\solextenda \cup \solextendb$. By \Cref{d-comp}, the set of $d$-invariants of self-conjugate spin$^c$ structures on $\Sol$ is therefore: 
$$
\left\{\tfrac{-b- 2}{4},\tfrac{-b- 2}{4},\tfrac{-b+ 2}{4},\tfrac{-b+ 2}{4},\tfrac{a- 2}{4}, \tfrac{a- 2}{4}, \tfrac{a+2}{4}, \tfrac{a+2}{4}\right\}
$$
and similarly for $\Solprime$.    
Without loss of generality we assume $b\leq b' $ (and so $-a'\leq -a$ by \Cref{d-sum-sol}).

\begin{enumerate}[label=Case \Roman*:,align=left]
\item $b'+2=b+2$ this is equivalent to $b'=b$ and $a'=a$ so that $\Sol=\Solprime$.
\item $b'+2=-a+2$ this implies $b'=-a$ and $a'=-b$ so that $\Solprime=M_{-b,-a}=\Sol$ by \Cref{woop}.
\item $b'+2=-a-2$ implies $b'+2-a'=-a-2-a'$, and since $a'-b'=a-b$, this gives $-a'-2=b+2$. Since $b'-2<b'+2=-a-2<-a+2$, the value $b'-2$ has to be equal to $b-2$ which immediately shows $\Sol=\Solprime$. 
\item $b'+2$ cannot be equal to $b-2$ since by assumption $b\leq b'$. 
\end{enumerate}
 
Finally, we have the case that all parameters $a, b, a', b'$ are even, which requires a more subtle argument.  
In this case, the self-conjugate spin$^c$ structures on $\Sol$ correspond to $\solextend$ by \Cref{lem:extend} and hence are 0 by \Cref{d-comp}.  Similarly for $\Solprime$, so we need to compare the non self-conjugate spin$^c$ structures.  From \Cref{lem:extend} and \Cref{d-comp}, we can describe the $d$-invariants on $\Sol$ by 
$$
\left\{\boxed{\tfrac{-b- 2}{4},\tfrac{-b- 2}{4},\tfrac{-b+ 2}{4},\tfrac{-b+ 2}{4}},\boxed{\tfrac{a- 2}{4}, \tfrac{a- 2}{4}, \tfrac{a+2}{4}, \tfrac{a+2}{4}}, \boxed{q_1,q_2,q_3,q_4}\right\},
$$
where the three blocks of $d$-invariants on $\Sol$ are grouped by non-zero value of $c_1$.  We have a similar statement for $\Solprime$.  If $\mathfrak{s}_1$ and $\mathfrak{s}_2$ on $\Sol$ have the same first Chern class, then the corresponding spin$^c$ structures on $\Solprime$ under the homology cobordism have the same first Chern class as well.   This means that there must be an identification of the above blocks for $\Sol$ and $\Solprime$.  Note that the identification of $H^2(\Sol)$ with $\Z/4 \oplus \Z/4$ is non-canonical, so we do not know a priori how the triples of blocks for $\Sol$ and $\Solprime$ are supposed to match under the homology cobordism.

There are three subcases to consider.
\begin{enumerate}[label=Case \Roman*:,align=left]
\item The $b$-block and $b'$-block correspond.  
This only happens if $b'=b$.  Then \Cref{d-sum-sol} implies $a'=a$.  
\item The $b$-block and $a'$-block correspond.  
This implies that $-b=a'$ and so $a=-b'$ by \Cref{d-sum-sol}.  Therefore, $\Sol$ and $\Solprime$ agree by \Cref{woop}.  
\item The $b$-block and $q'$-block correspond.  
Then the $a$-block corresponds to either the $a'$-block or $b'$-block.  As above, we have that $\Sol$ and $\Solprime$ agree.  
\end{enumerate}
\end{proof}

\appendix
\section{Diagrams}\label{diagrams}
Our arguments rely on a careful analysis of the rational homology cobordisms relating $\Sol$ with dihedral manifolds. More precisely, they rely on a detailed understanding of the following diagram:
\begin{equation}\label{diag:H2}
\xymatrix@R=0ex{
&H^2(W_{-b})\ar[r]^\db\ar[dl]_\wb&H^2(D_{-b})\\
H^2(\Sol)&&\\
&H^2(W_a)\ar[r]^\da\ar[ul]_\wa&H^2(D_a)\\
}\end{equation}
In this appendix we include the different versions of \Cref{diag:H2} according to the parities of the parameters of $\Sol$. We remark that the case $a\equiv 1,\, b\equiv 0 \pmod{2}$ is analogous to the case $a\equiv 0,\, b\equiv 1 \pmod{2}$ and so we skip the former.

\subsection{$a, b \equiv 0 \pmod{2}$}
$$\xymatrix@R=0ex@C=0.7in{
&\Z/4\oplus\Z/2\ar[r]^{\db=\left(\begin{smallmatrix}
1 & 0 \\
0 & 1
\end{smallmatrix}\right)}\ar[dl]_{\wb=\left(\begin{smallmatrix}
1 & 0 \\
0 & 2
\end{smallmatrix}\right) }&\Z/2\oplus\Z/2\\
\Z/4\oplus\Z/4&&\\
&\Z/2\oplus\Z/4\ar[r]_{\da=\left(\begin{smallmatrix}
1 & 0 \\
0 & 1
\end{smallmatrix}\right)}\ar[ul]^{\wa=\left(\begin{smallmatrix}
2& 0 \\
0 & 1
\end{smallmatrix}\right)
}&\Z/2\oplus\Z/2\\
}$$

\subsection{$a\equiv 0,\, b\equiv 1 \pmod{2}$}
$$\xymatrix@R=0ex@C=0.7in{
&\Z/2\oplus\Z/4\ar[r]^{\db=\left(\begin{smallmatrix}
0 &1
\end{smallmatrix}\right) }\ar[dl]_{ \wb=\left(\begin{smallmatrix}
2 & 0 \\
0 & 1
\end{smallmatrix}\right) }&\Z/4\\
\Z/4\oplus\Z/4&&\\
&\Z/2\oplus\Z/4\ar[r]_{\da=\left(\begin{smallmatrix}
1 & 0 \\
0 & 1
\end{smallmatrix}\right)}\ar[ul]^{\wa=\left(\begin{smallmatrix}
2 & 1 \\
0 & 1
\end{smallmatrix}\right)}&\Z/2\oplus\Z/2\\
}$$

\subsection{$a, b \equiv 1 \pmod{2}$}
$$\xymatrix@R=0ex@C=0.7in{
&\Z/2\oplus\Z/4\ar[r]^{\db=\left(\begin{smallmatrix}
0 &1
\end{smallmatrix}\right) }\ar[dl]_{\wb=\left(\begin{smallmatrix}
0 & 0  \\
1& 0 \\
0&1
\end{smallmatrix}\right) }&\Z/4\\
\Z/2\oplus\Z/2\oplus \Z/4&&\\
&\Z/2\oplus\Z/4\ar[r]_{\da=\left(\begin{smallmatrix}
2 & -1
\end{smallmatrix}\right)}\ar[ul]^{\wa=\left(\begin{smallmatrix}
0 & 1  \\
1& 0 \\
0&1
\end{smallmatrix}\right) }&\Z/4\\
}$$

\bibliographystyle{plain}
\bibliography{sol_refs}
\end{document}